\documentclass{amsart}

\pagespan{224}{228}
\makeatletter
\def\pageinfo{Pages \start@page--\end@page{} (June 24, 2024)}
\makeatother
\copyrightinfo{2024}{by the author(s) under Creative Commons Attribution 3.0 License (CC BY 3.0)}
\PII{\url{https://doi.org/10.1090/bproc/219}}

\makeatletter
\def\@maketitle{%
  \normalfont\normalsize
  \let\@makefnmark\relax  \let\@thefnmark\relax
  \ifx\@empty\@date\else \@footnotetext{\@setdate}\fi
  \ifx\@empty\@subjclass\else \@footnotetext{\@setsubjclass}\fi
  \ifx\@empty\@keywords\else \@footnotetext{\@setkeywords}\fi
  \ifx\@empty\thankses\else \@footnotetext{%
    \def\par{\let\par\@par}\@setthanks}\fi
  \global\topskip42\p@ % 5.5 picas to the base of the first title line
  \@settitle
  \ifx\@empty\authors \else \@setauthors \fi
  \ifx\@empty\@commby
  \else
    \baselineskip18\p@
    \vtop{\centering{\footnotesize\@commby\@@par}%
      \global\dimen@i\prevdepth}\prevdepth\dimen@i
  \fi
  \ifx\@empty\@dedicatory
  \else
    \baselineskip18\p@
    \vtop{\centering{\footnotesize\itshape\@dedicatory\@@par}%
      \global\dimen@i\prevdepth}\prevdepth\dimen@i
  \fi
  \@setabstract
  \normalsize
  \dimen@24\p@ \advance\dimen@-\baselineskip
  \vskip\dimen@\relax
} % end \@maketitle
\makeatother

\usepackage[pdftex,colorlinks,citecolor=black,linkcolor=black,urlcolor=black,bookmarks=false]{hyperref}

\def\MR#1{\href{https://www.ams.org/mathscinet-getitem?mr=#1}{MR#1}}

\def\doi#1{DOI \href{https://doi.org/#1}{#1}}

\newtheorem{theorem}{Theorem}[section]

\newtheorem*{question*}{Question}

\theoremstyle{definition}

\newtheorem{problem}[theorem]{Problem}

\newcommand{\R}{\mathbb{R}}
\newcommand{\C}{\mathbb{C}}

\numberwithin{equation}{section}

\title[Sign uncertainty principles and low-degree polynomials]{Sign uncertainty principles\\ and low-degree polynomials}

\author{Henry Cohn}
\address{Microsoft Research New England, One Memorial Drive, Cambridge, Massachusetts
02142}
\email{cohn@microsoft.com}

\author{Dingding Dong}
\address{Department of Mathematics, Harvard University, Cambridge, Massachusetts 02138}
\email{ddong@math.harvard.edu}

\author{Felipe Gon\c{c}alves}
\address{Instituto Nacional de Matem\'atica Pura e Aplicada, Rio de Janeiro, Brazil}
\email{goncalves@impa.br}

\date{October 17, 2023}
\subjclass[2020]{Primary 42B10, 42C05; Secondary 52C17.}

\commby{Dmitriy Bilyk}

\begin{document}
 
\begin{abstract}
We prove an asymptotically sharp version of the Bourgain--Clozel--Kahane and Cohn--Gon\c{c}alves sign uncertainty principles for polynomials of sublinear degree times a Gaussian, as the dimension tends to infinity. In particular, we show that polynomials whose degree is sublinear in the dimension cannot improve asymptotically on those of degree at most three. This question arises naturally in the study of both linear programming bounds for sphere packing and the spinless modular bootstrap bound for free bosons.
\end{abstract}

\maketitle

\section{Introduction}
\label{sec:intro}

Suppose we normalize the Fourier transform $\widehat{f}$ of an integrable function $f \colon \R^d \to \C$ by
\[
\widehat{f}(y) = \int_{\R^d} f(x) e^{-2\pi i \langle x,y \rangle}\, dx.
\]
Which properties of $f$ and $\widehat{f}$ are consistent, in the sense that there exists a function $f$ for which $f$ and $\widehat{f}$ behave as specified? This question is a fundamental research topic in harmonic analysis. Often, there is a trade-off between how $f$ and $\widehat{f}$ can behave. For example, the Heisenberg uncertainty principle says roughly that if $f$ is tightly concentrated near a point (i.e., its values decay quickly away from that point), then $\widehat{f}$ must be more dispersed. More generally, an uncertainty principle is any theorem that expresses a similar trade-off.

Bourgain, Clozel, and Kahane \cite{BCK10} discovered a beautiful uncertainty principle for the signs of functions. It deals with functions for which both $f$ and $\widehat{f}$ are real-valued; specifically, suppose $f \colon \R^d \to \R$ is integrable, $\widehat{f}$ is real-valued (equivalently, $f$ is an even function), and $\widehat{f}$ is integrable. Furthermore, we impose the following sign conditions on $f$ and $\widehat{f}$ for some nonnegative real numbers $\rho_1$ and $\rho_2$:
\begin{enumerate}
\item $f(x) \ge 0$ whenever $|x| \ge \rho_1$, while $\widehat{f}(0) \le 0$, and

\item $\widehat{f}(y) \ge 0$ whenever $|y| \ge \rho_2$, while $f(0) \le 0$.
\end{enumerate}
Assuming $f$ does not vanish almost everywhere, $\rho_1$ and $\rho_2$ must be positive; for example, if $\rho_1=0$, then $\widehat{f}(0) = \int_{\R^d} f(x) \, dx \ge 0$, and therefore $\widehat{f}(0)=0$, which implies that $f$ vanishes almost everywhere. Either $\rho_1$ or $\rho_2$ can be made arbitrarily close to zero on its own, but
Bourgain, Clozel, and Kahane proved that $\rho_1\rho_2$ is bounded away from $0$, and, in fact, that $\rho_1\rho_2 \ge cd$ for a positive constant $c$. 

The optimal bound is not known in any dimension except $d=12$, where Cohn and Gon\c{c}alves \cite{CG19} showed that the minimal possible value of $\rho_1\rho_2$ is $2$ by using techniques that originated in Viazovska's work on sphere packing \cite{CKMRV17,V17}. Cohn and Gon\c{c}alves also studied a related uncertainty principle, which replaces the inequalities on $f$ and $\widehat{f}$ with
\begin{enumerate}
\item $f(x) \ge 0$ whenever $|x| \ge \rho_1$, while $\widehat{f}(0) \le 0$, and

\item $\widehat{f}(y) \le 0$ whenever $|y| \ge \rho_2$, while $f(0) \ge 0$.
\end{enumerate}
This variant turns out to be even more closely related to sphere packing bounds. It has been solved exactly for $d=1$, $8$, and $24$, with the minimal value of $\rho_1\rho_2$ being $1$, $2$, and $4$, respectively. In fact, these bounds are implied by the sharpness of the linear programming bound for sphere packing in these dimensions. Furthermore, the answer is conjectured to be $2/\sqrt{3}$ for $d=2$, and it is known to agree with $2/\sqrt{3}$ to at least one thousand decimal places \cite{ACHLT}. See \cite{ACHLT,CG19} for numerical computations, as well as 
\cite{GOR} for a different numerical technique, which leads to the remarkable conjecture that $\rho_1\rho_2=1/(1+\sqrt{5})$ for $d=1$ in the Bourgain--Clozel--Kahane case.

These sign uncertainty optimization problems can be reduced to the following problem on eigenfunctions of the Fourier transform \cite{BCK10,CG19}, with $\rho=\sqrt{\rho_1\rho_2}$. Here $s=1$ corresponds to the Bourgain--Clozel--Kahane case, while $s=-1$ corresponds to the other variant.

\begin{problem} \label{problem:eigenvalue}
Let $s \in \{\pm1\}$, and let $d$ be a positive integer. What is the smallest real number $\rho \ge 0$ for which there exists a radial, integrable function $f \colon \R^d \to \R$ such that
\begin{enumerate}
\item $f$ is not identically zero,

\item $\widehat{f} = s f$, 

\item $f(0)=0$, and

\item $f(x) \ge 0$ whenever $|x| \ge \rho$?
\end{enumerate}
\end{problem}

Cohn and Gon\c{c}alves conjectured that the limit of $\rho/\sqrt{d}$ as $d \to \infty$ exists and is independent of the choice of $s$ \cite[Conjecture~1.5]{CG19}. Determining the value of this limit is a key open problem in this area. Conjecture~3.2 in \cite{ACHLT} predicts that $\rho \sim \sqrt{d}/\pi$ as $d \to \infty$, which would improve on all existing bounds. However, the available numerics yield only three or four digits of accuracy for the limit, which makes it difficult to predict the exact answer with confidence.

In the absence of an exact solution, it is natural to study this problem numerically. One particularly tractable family of test functions is polynomials times Gaussians: we can take $f(x) = p(2 \pi |x|^2) e^{-\pi |x|^2}$, where $p$ is a single-variable polynomial.  To compute the Fourier transform of $f$, it is convenient to write $p$ as a linear combination of the Laguerre polynomials $L_k^{d/2-1}$ of degree $k$ with parameter $d/2-1$, because $x \mapsto L_k^{d/2-1}(2\pi |x|^2) e^{-\pi |x|^2}$ is an eigenfunction of the Fourier transform with eigenvalue $(-1)^k$. Thus, the Fourier transform $\widehat{f}$ is also given by $\widehat{f}(y) = q(2 \pi |y|^2) e^{-\pi |y|^2}$ for some polynomial $q$, namely the corresponding linear combination of $(-1)^k L_k^{d/2-1}$. 

Setting $f(x) = p(2 \pi |x|^2) e^{-\pi |x|^2}$ with a suitably chosen polynomial $p$ of degree at most three yields $\rho \sim \sqrt{d/(2\pi)}$ (see \cite{BCK10,CG19}). Our main theorem in this paper says that we obtain the same asymptotics as $d \to \infty$ whenever the degree of $p$ is sublinear in $d$. In particular, our result proves Conjecture~3.2 in \cite{CG19}, which is the special case when the degree of $p$ is bounded.

\begin{theorem} \label{theorem:intro}
Let $\rho_{d,s,n}$ be the optimal value of $\rho$ in Problem~\ref{problem:eigenvalue} when $f$ is restricted to be of the form $f(x) = p(2 \pi |x|^2) e^{-\pi |x|^2}$ with $p$ a polynomial of degree at most $n$.  Then for each  $s = \pm 1$ and every function $d \mapsto n(d) \in \{3, 4, 5, \dots\}$ with $\lim_{d \to \infty} n(d)/d = 0$,
\[
\rho_{d,s,n(d)} \sim \sqrt{\frac{d}{2\pi}} 
\]
as $d \to \infty$.
\end{theorem}

This theorem shows that to obtain informative numerics in high dimensions, we must use high-degree polynomials. It would be interesting to have a more precise quantitative relationship between the dimension, polynomial degree, and quality of the optimized bound. We do not even have a proof that $\rho_{d,s,n}$ converges to the optimal value of $\rho$ in Problem~\ref{problem:eigenvalue} as $n \to \infty$ with $d$ fixed, although we expect that is true.

When $s=-1$ and $d$ is even, Problem~\ref{problem:eigenvalue} arises in theoretical physics \cite{HMR} as the spinless modular bootstrap bound with $U(1)^c$ symmetry, where $c=d/2$. It gives an upper bound of $\rho^2/2$ for the spectral gap in two-dimensional conformal field theories describing $c$ free bosons, with the most interesting case being the limit as $c \to \infty$. While most conformal field theories of interest in physics are not of this form, free theories are a noteworthy special case; for example, they provide a natural candidate for a holographic dual to an exotic theory of three-dimensional gravity \cite{ACHT,MW}. Polynomials times Gaussians are the primary choice of test function for numerical computations in this area (see, for example, \cite{AHT}), and our theorem therefore sets limits on the feasibility of these computations for large $c$.

Our approach to proving Theorem~\ref{theorem:intro} was inspired by a paper of Friedan and Keller \cite{FK}, who gave a heuristic derivation of a result analogous to Theorem~\ref{theorem:intro} for bounded-degree polynomials in a closely related problem, namely the spinless modular bootstrap assuming only Virasoro symmetry (see  \cite[Section~3.2]{FK}). We do not know how to produce a rigorous proof directly from their techniques, but our proof was motivated by their work.

\section{Proof of Theorem~\ref{theorem:intro}}
\label{sec:proof}

We will deduce Theorem~\ref{theorem:intro} from a slightly stronger bound, which is not restricted to eigenfunctions:

\begin{theorem} \label{theorem:main}
Suppose $f \colon \R^d \to \R$ is given by $f(x) = p(2\pi |x|^2) e^{-\pi |x|^2}$, where $p$ is a polynomial of degree at most $n$ and $p$ is not identically zero. If $\widehat{f}(0)=0$, then $p$ has a sign change at or beyond $2\lambda$, where $\lambda$ is the smallest root of the Laguerre polynomial $L_{m}^{d/2-1}$ of degree $m = \lfloor n/2 \rfloor + 1$. 
\end{theorem}

By \cite[Theorem~4]{IL92},\footnote{In our notation, Theorem~4 states that $\lambda > 2m + d/2-3 - \sqrt{1 + a(m-1)(m+d/2-2)}$ if $a > 4 \cos^2(\pi/(m+1))$, but that lower bound for $\lambda$ is not always true. In the notation used in \cite{IL92}, the issue in the proof is that the minimum of $y_n$ over $0<n<N$ may not occur at $n=N-1$. However, it is true for $a=4$: in that case one can check that $y_n$ is a decreasing function of $n$ when $\alpha>1$, while for $\alpha\le1$ this lower bound for the smallest Laguerre polynomial root is nonpositive and thus automatically holds.}
\begin{equation} \label{eq:lambda}
\lambda > 2m + d/2-3 - \sqrt{1 + 4 (m-1)(m+d/2-2)}.
\end{equation}
In particular, if $d \to \infty$ with $n = o(d)$ (and therefore $m = o(d)$ as well), then $2\lambda \ge (1-o(1))d$. If $f(x)$ has the same sign for all $x$ with $|x| \ge \rho$, then Theorem~\ref{theorem:main} yields $2\pi\rho^2 \ge 2\lambda \ge (1-o(1))d$, and Theorem~\ref{theorem:intro} follows. Specifically, this argument proves a lower bound for $\rho$, while the matching upper bound can already be achieved when $n=3$.

\begin{proof}[Proof of Theorem~\ref{theorem:main}]
We begin by observing that $\widehat{f}(0)=0$ means
\[
\int_{\R^d} p(2\pi |x|^2) e^{-\pi |x|^2} \, dx = 0.
\]
If we change to a radial variable $r = |x|$, this equation amounts to
\[
\int_0^\infty p(2\pi r^2) e^{-\pi r^2} r^{d-1} \, dr = 0,
\]
or equivalently
\[
\int_0^\infty p(2u) e^{-u} u^{d/2-1} \, du = 0
\]
after setting $u = \pi r^2$. Note that the Laguerre polynomials $L_k^{d/2-1}$ are orthogonal with respect to the measure $e^{-u} u^{d/2-1} \, du$ on $[0,\infty)$ that occurs in this equation.

Now we can apply Gauss quadrature \cite[Theorem~5.3.2]{AAR}. For the measure $e^{-u} u^{d/2-1} \, du$ on $[0,\infty)$, Gauss quadrature says that if $u_1<\dots<u_m$ are the roots of $L_m^{d/2-1}$, then there are positive weights $w_1,\dots,w_m$ such that for every polynomial $q$ of degree at most $2m-1$,
\[
\int_0^\infty q(u) e^{-u} u^{d/2-1} \, du = \sum_{i=1}^m w_i q(u_i).
\]
In particular, we can take $m = \lfloor n/2 \rfloor + 1$ as above and $q(u) = p(2u)$. Then $2m-1 \ge n \ge \deg(p)$, and therefore
\[
\sum_{i=1}^m w_i p(2u_i) = 0.
\]
If $p(2u_1),\dots,p(2u_m)$ all have the same sign, then they must all vanish, because the weights $w_1,\dots,w_m$ are positive. These roots of $p$ cannot all have multiplicity of at least $2$, because $p$ would then have at least $2m$ roots and would vanish identically. Thus, $p$ must have a sign change at some point greater than or equal to $2u_1$, which is $2\lambda$ in the notation of Theorem~\ref{theorem:main}. This completes the proof of Theorem~\ref{theorem:main}, and thus Theorem~\ref{theorem:intro} as well.
\end{proof}

In light of Theorem~\ref{theorem:intro}, it is natural to ask how well one can do in Problem~\ref{problem:eigenvalue} using polynomials of degree growing linearly in the dimension. In other words, using  functions of the form $f(x) = p(2 \pi |x|^2) e^{-\pi |x|^2}$ with $p$ a polynomial of degree $n = (c+o(1))d$ as $d \to \infty$ for some constant $c>0$. Combining Theorem~\ref{theorem:main} with the bound \eqref{eq:lambda} shows that the radius of the last sign change must satisfy
\[
\rho \ge \left(\sqrt{\frac{c + 1/2 - \sqrt{c(c+1)}}{\pi}} - o(1)\right) \sqrt{d}
\]
as $d \to \infty$. For example, to have any hope of proving the conjecture that $\rho=(1/\pi + o(1))\sqrt{d}$ from \cite{ACHLT}, $c$ must be at least $(\pi-2)^2/(8\pi) = 0.05185402\dots$. This lower bound is rather weak, and we expect that it is far from the truth. Based on the data from \cite[Appendix~B]{ACHLT}, we conjecture that $c$ must tend to infinity to obtain the optimal asymptotics. Theorem~\ref{theorem:main} is incapable of proving such a bound, because $\lambda \to 0$ as $c \to \infty$ (see, for example, \cite[(6.36.6)]{Sz75}).

\end{document}